\documentclass[9pt,twoside,reqno]{amsart}
\usepackage{amsfonts}
\usepackage{hyperref}

\newtheorem{theorem}{Theorem}
\theoremstyle{plain}

\newtheorem{corollary}{Corollary}

\newtheorem{definition}{Definition}
\newtheorem{example}{Example}

\newtheorem{lemma}{Lemma}

\newtheorem{proposition}{Proposition}
\newtheorem{remark}{Remark}

\numberwithin{equation}{section}
\input{tcilatex}

\begin{document}
\title[]{Partial $b_{v}\left( s\right) $ and $b_{v}\left( \theta \right) $
metric spaces and related fixed point theorems}
\author{Ibrahim Karahan}
\address{Department of Mathematics, Faculty of Science, Erzurum Technical
University, Erzurum, 25700, Turkey.}
\email{ibrahimkarahan@erzurum.edu.tr}
\author{Irfan Isik}
\address{Department of Mathematics, Faculty of Science, Erzurum Technical
University, Erzurum, 25700, Turkey.}
\email{iirfansk25@gmail.com}

\begin{abstract}
In this paper, we introduced two new generalized metric spaces called
partial $b_{v}\left( s\right) $ and $b_{v}\left( \theta \right) $ metric
spaces which extend $b_{v}\left( s\right) $ metric space, $b$-metric space,
rectangular metric space, $v$-generalized metric space, partial metric
space, partial $b$-metric space, partial rectangular $b$-metric space and so
on. We proved some famous theorems such as Banach, Kannan and Reich fixed
point theorems in these spaces. Also, we give definition of partial $v$%
-generalized metric space and show that these fixed point theorems are valid
in this space. We also give numerical examples to support our definitions.
Our results generalize several corresponding results in literature.
\end{abstract}

\maketitle

\section{Introduction and Preliminaries}

Metric space was introduced by Maurice Fr\'{e}chet \cite{metric} in 1906.
Since a metric induces topological properties, it has very large application
area in mathematics, especially in fixed point theory. Generalizing of
notions is in the nature of mathematics. So, after the notion of metric
space, many different type generalized metric spaces were introduced by many
researchers. In 1989, Bakhtin introduced the notion of $b$-metric spaces by
adding a multiplier to triangle ineuality. In 1994, Matthews \cite{mat}
introduced the notion of partial metric spaces. In this kind of spaces,
self-distance of any point need not to be zero. This space is used in the
study of denotational semantics of dataflow network. In 2000, Branciari \cite%
{bri} introduced rectangular metric space by adding four points instead of
three points in triangle inequality. These three spaces are the basis of
other generalized metric spaces. After all these spaces, $v$-generalized
metric space \cite{bri}, rectangular $b$-metric spaces \cite{geo}, $%
b_{v}\left( s\right) $ metric space \cite{mr}, partial $b$-metric space \cite%
{Shukla} and partial rectangular $b$-metric space \cite{shukla-2} were
introduced in recent years. Below, we give definitions of some generalized
metric spaces.

\begin{definition}
\cite{Bakhtin} Let $E$ be a nonempty set and $\rho :E\times E\rightarrow
\lbrack 0,\infty )$ a function. $\left( E,\rho \right) $ is called $b$%
-metric space if there exists a real number $s\geq 1$ such that following
conditions hold for all $u,w,v\in E$:

\begin{enumerate}
\item[(1)] $\rho (u,w)=0$ iff $u=w$;

\item[(2)] $\rho (u,w)=\rho (w,u)$;

\item[(3)] $\rho (u,w)\leq s[\rho (u,v)+\rho (v,w)].$
\end{enumerate}
\end{definition}

Clearly a $b$-metric space with $s=1$ is exactly a usual metric space.

\begin{definition}
\cite{mat} Let $E$ be a nonempty set and $\rho :E\times E\rightarrow \lbrack
0,\infty )$ a mapping. $\left( E,\rho \right) $ is called partial metric
space if following conditions hold for all $u,w,v\in E$:

\begin{enumerate}
\item $u=w$ iff $\rho (u,u)=\rho (u,w)=\rho (w,w)$;

\item $\rho (u,u)\leq \rho (u,w);$

\item $\rho (u,w)=\rho (w,u);$

\item $\rho (u,w)\leq \rho (u,v)+\rho (v,w)-\rho (v,v)$.
\end{enumerate}
\end{definition}

It is clear that every metric space is also a partial metric spaces.

\begin{definition}
\cite{bri} Let $E$ be a nonempty set and let $\rho :$ $E\times E\rightarrow
\lbrack 0,\infty )$ be a mapping. $(E,\rho )$ is called a rectangular metric
space if following conditions hold for all $u,w\in E$ and for all distinct
points $c,d\in E\setminus \left\{ u,w\right\} $:

\begin{enumerate}
\item $\rho (u,w)=0$ iff $u=w$;

\item $\rho (u,w)=\rho (w,u)$;

\item $\rho (u,w)\leq \rho (u,c)+\rho (c,d)+\rho (d,w)$.
\end{enumerate}
\end{definition}

\begin{definition}
\cite{Shukla} Let $E$ be a nonempty set and mapping $\rho :E\times
E\rightarrow \lbrack 0,\infty )$ a mapping. $\left( E,\rho \right) $ is
called partial $b$-metric space if there exists a real number $s\geq 1$ such
that following conditions hold for all $u,w,v\in E$:

\begin{enumerate}
\item $u=w$ iff $\rho (u,u)=\rho (u,w)=\rho (w,w)$;

\item $\rho (u,u)\leq \rho (u,w);$

\item $\rho (u,w)=\rho (w,u);$

\item $\rho (u,w)\leq s[\rho (u,v)+\rho (v,w)]-\rho (v,v)$.
\end{enumerate}
\end{definition}

\begin{remark}
\cite{Shukla} It is clear that every partial metric space is a partial $b$%
-metric space with coefficient $s=1$ and every $b$-metric space is a partial 
$b$-metric space with the same coefficient and zero self-distance. However,
the converse of this fact need not hold.
\end{remark}

In 2017, Mitrovic and Radenovic introduced following generalized metric
space which is referred to as $b_{v}(s)$ metric space. Under the suitable
assumptions, this kind of spaces can be reduced to the other spaces.

\begin{definition}
\cite{mr} Let $E$ be a nonempty set, $\rho :$ $E\times E\rightarrow \lbrack
0,\infty )$ a mapping and $v\in 
\mathbb{N}
$. Then $(E,\rho )$ is said to be a $b_{v}(s)$ metric space if there exists
a real number $s\geq 1$ such that following conditions hold for all $u,w\in
E $ and for all distinct points $z_{1},z_{2},\ldots ,z_{v}\in E\setminus
\left\{ u,w\right\} $:

\begin{enumerate}
\item[1.] $\rho (u,w)=0$ iff $u=w$;

\item[2.] $\rho (u,w)=\rho (w,u)$;

\item[3.] $\rho (u,w)\leq s[\rho (u,z_{1})+\rho (z_{1},z_{2})+\cdots +\rho
\left( z_{v},w\right) ].$
\end{enumerate}
\end{definition}

This metric space can be reduced to $v$-generalized metric space by taking $%
s=1$, rectangular metric space by taking $v=2$ and $s=1$, rectangular $b$%
-metric space by taking $v=2$, $b$-metric space by taking $v=1$ and usual
metric space by taking $v=s=1.$

\section{Main Results}

In this part, motivated and inspired by mentioned studies, we introduce $%
b_{v}\left( \theta \right) $ (or extended $b_{v}(s)$) metric space and
partial $b_{v}(s)$ metric space. Also we give some fixed point theorems in
these spaces.

First we introduce partial $b_{v}\left( s\right) $ metric space and give
some properties of it.

\subsection{Partial $b_{v}\left( s\right) $ Metric Spaces}

\begin{definition}
\label{Def}Let $E$ be a nonempty set and $\rho :$ $E\times E\rightarrow
\lbrack 0,\infty )$ be a mapping and $v\in 
\mathbb{N}
$. Then $(E,\rho )$ is said to be a partial $b_{v}(s)$ metric space if there
exists a real number $s\geq 1$ such that following conditions hold for all $%
u,w,z_{1},z_{2},\ldots ,z_{v}\in E$:
\end{definition}

\begin{enumerate}
\item $u=w\Leftrightarrow \rho (u,u)=\rho (u,w)=\rho (w,w)$;

\item $\rho (u,u)\leq \rho (u,w)$;

\item $\rho (u,w)=\rho (w,u)$;

\item $\rho (u,w)\leq s[\rho (u,z_{1})+\rho (z_{1},z_{2})+\ldots +\rho
(z_{v-1},z_{v})+\rho (z_{v},y)]-\sum_{i=1}^{v}\rho (z_{i},z_{i})$.
\end{enumerate}

It is easy to see that every $b_{v}(s)$\ metric space is a partial $b_{v}(s)$%
\ metric space. However, the converse is not true in general.

\begin{remark}
In Definition \ref{Def};

\begin{enumerate}
\item if we take $v=2$, then we derive partial rectangular $b$-metric space.

\item if we take $v=1,$\ then we derive partial $b$-metric space.

\item if we take $v=s=1$, then we derive partial metric space.
\end{enumerate}
\end{remark}

\begin{remark}
Let $(E,\rho )$ be a partial $b_{v}(s)$ metric space, if $\rho (u,w)=0$, for 
$u,w\in E$, then $u=w$.
\end{remark}

\begin{proof}
Let $\rho (u,w)=0$ for $u,w\in E$. From the second condition of partial $%
b_{v}(s)$ metric space, since $\rho (u,u)\leq \rho (u,w)=0$, we have $\rho
(u,u)=0$. Similarly, we have $\rho (w,w)=0$. So, we get $\rho (u,w)=\rho
(u,u)=\rho (w,w)=0$. It follows from the first condition that $u=w$.
\end{proof}

\begin{proposition}
Let $E$\ be a nonempty set such that $d_{1}$\ is a partial metric and $d_{2}$%
\ is a $b_{v}(s)$\ metric on $E$. Then $\left( E,\rho \right) $ is a partial 
$b_{v}(s)$\ metric space where $\rho :E\times E\rightarrow \lbrack 0,\infty
) $\ is a mapping defined by $\rho (u,w)=p(u,w)+d(u,w)$\ for all $u,w\in E.$
\end{proposition}

\begin{proof}
Let $(E,d_{1})$ be a partial metric space and $(E,d_{2})$ be a $b_{v}(s)$
metric space.Then it is clear that first three conditions of the partial $%
b_{v}(s)$ metric space are satisfied for the function $\rho $. Let $%
u,w,z_{1},z_{2},\ldots ,z_{v}\in E$ be arbitrary points. Then, we have%
\begin{eqnarray*}
\rho (u,w) &=&d_{1}(u,w)+d_{2}(u,w) \\
&\leq &d_{1}(u,z_{1})+d_{1}(z_{1},z_{2})+\ldots
+d_{1}(z_{v},w)-\sum_{i=1}^{v}d_{1}(z_{i},z_{i}) \\
&&+s\left[ d_{2}(u,z_{1})+d_{2}(z_{1},z_{2})+\ldots +d_{2}(z_{v},w)\right] \\
&\leq &s\left[ d_{1}(u,z_{1})+d_{1}(z_{1},z_{2})+\ldots
+d_{1}(z_{v},w)-\sum_{i=1}^{v}d_{1}(z_{i},z_{i})\right. \\
&&\left. +d_{2}(u,z_{1})+d_{2}(z_{1},z_{2})+\ldots +d_{2}(z_{v},w)\right] \\
&=&s\left[ \rho (u,z_{1})+\rho (z_{1},z_{2})+\ldots +\rho
(z_{v},w)-\sum_{i=1}^{v}\rho (z_{i},z_{i})\right] \\
&\leq &s\left[ \rho (u,z_{1})+\rho (z_{1},z_{2})+\ldots +\rho (z_{v},w)%
\right] -\sum_{i=1}^{v}\rho (z_{i},z_{i})\text{.}
\end{eqnarray*}

So, $(E,\rho )$ is a partial $b_{v}(s)$ metric space.
\end{proof}

Now, we give definitions of convergent sequence, Cauchy sequence and
complete partial $b_{v}(s)$ metric space by the following way.

\begin{definition}
Let $(E,\rho )$ be a partial $b_{v}(s)$ metric space and let $\left\{
u_{n}\right\} $ be any sequence in $E$ and $u\in E$. Then:
\end{definition}

\begin{enumerate}
\item The sequence $\left\{ u_{n}\right\} $ is said to be convergent and
converges to $u$, if $\lim_{n\rightarrow \infty }\rho (u_{n},u)=\rho (u,u)$.

\item The sequence $\left\{ u_{n}\right\} $ is said to be Cauchy sequence in 
$(E,\rho )$ if $\lim_{n,m\rightarrow \infty }\rho (u_{n},u_{m})$ exists and
is finite.

\item $(E,\rho )$ is said to be a complete partial $b_{v}(s)$ metric space
if for every Cauchy sequence $\left\{ u_{n}\right\} $ in $E$ there exists $%
u\in E$ such that%
\begin{equation*}
\lim_{n,m\rightarrow \infty }\rho (u_{n},u_{m})=\lim_{n\rightarrow \infty
}\rho (u_{n},u)=\rho (u,u)\text{.}
\end{equation*}%
Note that the limit of a convergent sequence may not be unique in a partial $%
b_{v}(s)$ metric space.
\end{enumerate}

Now we give an analogue of Banach contraction principle. Our proof is very
different from the original proof of Banach contraction principle in usual
metric space.

\begin{theorem}
\label{A}Let $(E,\rho )$ be a complete partial $b_{v}\left( s\right) $
metric space and $S:E\rightarrow E$ be a contraction mapping, i.e., $S$
satisfies%
\begin{equation}
\rho (Su,Sw)\leq \lambda \rho (u,w)  \label{1}
\end{equation}%
for all $u,w\in E,$ where $\lambda \in \lbrack 0,1)$. Then $S$ has a unique
fixed point $b\in S$ and $\rho (b,b)=0$.
\end{theorem}

\begin{proof}
Let $G=S^{n_{0}}$ and define a sequence $\left\{ u_{n}\right\} $ by $%
Gu_{n}=u_{n+1}$ for all $n\in 
\mathbb{N}
$ and arbitrary point $u_{0}\in E$. Since $\lambda \in \lbrack 0,1)$ and $%
\lim_{n\rightarrow \infty }\lambda ^{n}=0$, there exists a natural number $%
n_{0}$ such that $\lambda ^{n_{0}}<\frac{\varepsilon }{4s}$\ for given $%
0<\varepsilon <1.$ Then, for all $u,w\in E$ we get%
\begin{equation}
\rho (Gu,Gw)=\rho (S^{n_{0}}u,S^{n_{0}}w)\leq \lambda ^{n_{0}}\rho (u,w)%
\text{.}  \label{2}
\end{equation}%
So, we have%
\begin{equation*}
\rho (u_{k+1},u_{k})=\rho (Gu_{k},Gu_{k-1})\leq \lambda ^{n_{0}}\rho
(u_{k},u_{k-1})\leq \lambda ^{kn_{0}}\rho (u_{1},u_{0})\rightarrow 0\text{,}
\end{equation*}%
as $k\rightarrow \infty $. Hence, there exists a $l\in 
\mathbb{N}
$ such that%
\begin{equation*}
\rho (u_{l+1},u_{l})<\frac{\varepsilon }{4s}\text{.}
\end{equation*}%
Now, let%
\begin{equation*}
B_{\rho }[u_{l},\varepsilon /2]:=\left\{ w\in E:\rho (u_{l},w)\leq \frac{%
\varepsilon }{2}+\rho (u_{l},u_{l})\right\} \text{.}
\end{equation*}%
We need to prove that $G$ maps the set $B_{\rho }[u_{l},\varepsilon /2]$
into itself. Since $u_{l}\in B_{\rho }[u_{l},\varepsilon /2]$, it is a
nonempty set. Let $z$ be an arbitrary point in $B_{\rho }[u_{l},\varepsilon
/2]$. Then, using (\ref{2}) we get%
\begin{eqnarray*}
\rho (Gz,u_{l}) &\leq &s\left[ \rho (Gz,Gu_{l})+\rho
(Gu_{l}+Gu_{l+1})+\ldots +\rho (Gu_{l+v-2},Gu_{l+v-1})\right.  \\
&&\left. +\rho (Gu_{l+v-1},u_{l})\right] -\sum\limits_{i=0}^{v-1}\rho
(Gu_{l+i},Gu_{l+i}) \\
&\leq &s\left[ \rho (Gz,Gu_{l})+\rho (Gu_{l}+Gu_{l+1})+\ldots +\rho
(Gu_{l+v-2},Gu_{l+v-1})+\rho (Gu_{l+v-1},u_{l})\right]  \\
&\leq &s[\lambda ^{n_{0}}(\frac{\varepsilon }{2}+\rho (u_{l},u_{l}))+\rho
(u_{l+1},u_{l+2})+\ldots +\rho (u_{l+v-1},u_{l+v})+\rho (u_{l+v},u_{l}) \\
&\leq &s\left\{ \lambda ^{n_{0}}(\frac{\varepsilon }{2}+\rho
(u_{l},u_{l}))+\rho (u_{l+1},u_{l+2})+\ldots +\rho
(u_{l+v-1},u_{l+v})+s[\rho (u_{l},u_{l+1})\right.  \\
&&\left. +\rho (u_{l+1},u_{l+2})+\ldots +\rho (u_{l+v-1},u_{l+v})+\rho
(u_{l+v},u_{l+v})]-\sum\limits_{i=1}^{v}\rho (u_{l+i},u_{l+i})\right\}  \\
&\leq &s\left\{ \lambda ^{n_{0}}(\frac{\varepsilon }{2}+\rho
(u_{l},u_{l}))+\left( s+1\right) \rho (u_{l},u_{l+1})+(s+1)\rho
(u_{l+1},u_{l+2})+\right.  \\
&&\left. (s+1)\rho (u_{l+2},u_{l+3})+\ldots +(s+1)\rho
(u_{l+v-1},u_{l+v})+s\rho (u_{l+v},u_{l+v})\right\}  \\
&\leq &s\left\{ \lambda ^{n_{0}}(\frac{\varepsilon }{2}+\rho
(u_{l},u_{l}))+(s+1)\rho (u_{l},u_{l+1})+(s+1)\rho (u_{l+1},u_{l+2})+\right. 
\\
&&\left. (s+1)\rho (u_{l+2},u_{l+3})+\ldots +(s+1)\rho
(u_{l+v-1},u_{l+v})+s\lambda ^{vn_{0}}\rho (u_{l},u_{l})\right\}  \\
&=&\rho (u_{l},u_{l})\left[ s\lambda ^{n_{0}}+s^{2}\lambda ^{vn_{0}}\right]
+s\lambda ^{n_{0}}\frac{\varepsilon }{2}+s^{2}\rho (u_{l},u_{l+1})+ \\
&&s(s+1)\left[ \rho (u_{l+1},u_{l+2})+\ldots +\rho (u_{l+v-1},u_{l+v})\right]
\end{eqnarray*}%
Since $\lambda ^{n_{0}}<\frac{\varepsilon }{4s}$ and $\rho
(u_{l},u_{l+1})\leq \frac{\varepsilon }{4v(s^{2}+s)},$ we have%
\begin{eqnarray*}
\rho (Gz,u_{l}) &\leq &\rho (u_{l},u_{l})\left[ s\frac{\varepsilon }{4s}%
+s^{2}\frac{\varepsilon ^{v}}{(4s)^{v}}\right] +s\frac{\varepsilon }{4s}%
\frac{\varepsilon }{2}+ \\
&&(s^{2}+s)\left[ \rho (u_{l},u_{l+1})+\rho (u_{l+1},u_{l+2})+\ldots +\rho
(u_{l+v-1},u_{l+v})\right]  \\
&\leq &\rho (u_{l},u_{l})+\frac{\varepsilon }{4}+(s^{2}+s)v\frac{\varepsilon 
}{4v(s^{2}+s)} \\
&=&\frac{\varepsilon }{2}+\rho (u_{l},u_{l})\text{.}
\end{eqnarray*}%
So, $Gz\in B_{\rho }[u_{l},\varepsilon /2]$. Therefore, $G$ maps $B_{\rho
}[u_{l},\varepsilon /2]$ into itself. Since $u_{l}\in B_{\rho
}[u_{l},\varepsilon /2]$ and $Gu_{l}\in B_{\rho }[u_{l},\varepsilon /2]$, we
obtain that $G^{n}u_{l}\in B_{\rho }[u_{l},\varepsilon /2]$ for all $n\in 
\mathbb{N}
$, that is, $u_{m}\in B_{\rho }[u_{l},\varepsilon /2]$ for all $m\geq l$. On
the other hand, from definition of partial $b_{v}\left( s\right) $ metric
space, since $\rho (u_{l},u_{l})\leq \rho (u_{l},u_{l+1})<\frac{\varepsilon 
}{4v(s^{2}+s)}<\frac{\varepsilon }{2}$, we have%
\begin{equation*}
\rho (u_{n},u_{m})<\frac{\varepsilon }{2}+\rho (u_{l},u_{l})<\varepsilon 
\end{equation*}%
for all $n,m>l$. This means that the sequence $\left\{ u_{n}\right\} $ is a
Cauchy sequence. Completeness of $E$ implies that there exists $b\in E$ such
that%
\begin{equation}
\lim_{n\rightarrow \infty }\rho (u_{n},b)=\lim_{n,m\rightarrow \infty }\rho
(u_{n},u_{m})=\rho (b,b)=0\text{.}  \label{3}
\end{equation}

Now, we need to show that, $b$ is a fixed point of $S$. For any $n\in 
\mathbb{N}
$ we get%
\begin{eqnarray*}
\rho (b,Sb) &\leq &s\left[ \rho (b,u_{n+1})+\rho (u_{n+1},u_{n+2})+\ldots
+\rho (u_{n+v-1},u_{n+v})+\right.  \\
&&\left. \rho (u_{n+v},Sb)\right] -\sum\limits_{i=1}^{v}\rho
(u_{n+i},u_{n+i}) \\
&\leq &s\left[ \rho (b,u_{n+1})+\rho (u_{n+1},u_{n+2})+\ldots +\rho
(u_{n+v-1},u_{n+v})+\rho (u_{n+v},Sb)\right]  \\
&\leq &s\left[ \rho (b,u_{n+1})+\rho (u_{n+1},u_{n+2})+\ldots +\rho
(u_{n+v-1},u_{n+v})+\rho (Su_{n+v-1},Sb)\right]  \\
&\leq &s\left[ \rho (b,u_{n+1})+\rho (u_{n+1},u_{n+2})+\ldots +\rho
(u_{n+v-1},u_{n+v})+\lambda \rho (u_{n+v-1},b)\right] .
\end{eqnarray*}%
So, it follows from (\ref{3}) that $\rho (b,Sb)=0$. So, $b$ is a fixed point
of $S$.

Now, we show that $S$ has a unique fixed point. Let $a,b\in E$ be two
distinct fixed points of $S$, that is, $Sa=a$, $Sb=b$. Then, contractivity
of mapping $S$ implies that%
\begin{equation*}
\rho (a,b)=\rho (Sa,Sb)\leq \lambda \rho (a,b)<\rho (a,b)\text{,}
\end{equation*}%
which is a contradiction. So, it folllows that $\rho (a,b)=0$, that is, $a=b$%
. Moreover, for a fixed point $a$, let assume that $\rho (a,a)>0$. Then we
get $\rho (a,a)=\rho (Sa,Sa)\leq \lambda \rho (a,a)<\rho (a,a)$ which is a
contradiction. So, we have $\rho (a,a)=0$.
\end{proof}

Now, we prove an analogue of Kannan fixed point theorem.

\begin{theorem}
\label{B}Let $(E,\rho )$ be a complete partial $b_{v}\left( s\right) $
metric space and $S:E\rightarrow E$ a mapping satisfying the following
condition:%
\begin{equation}
\rho (Su,Sy)\leq \lambda \left[ \rho (u,Su)+\rho (w,Sw)\right]   \label{4}
\end{equation}%
for all $u,w\in E$, where $\lambda \in \lbrack 0,\frac{1}{2})$, $\lambda
\neq \frac{1}{s}$. Then $S$ has a unique fixed point $b\in E$ and $\rho
(b,b)=0$.
\end{theorem}

\begin{proof}
$.$First we show the existence of fixed points of $S$. Let define a sequence 
$\left\{ u_{n}\right\} $ by $u_{n}=S^{n}u_{0}$  for all $n\in 
\mathbb{N}
$ and an arbitrary point $u_{0}\in E$ and $\sigma _{n}=\rho (u_{n},u_{n+1})$%
. If \ $\sigma _{n}=0$, \ then for at least one $n,$ $u_{n}$ is a fixed
point of $S$. So, let assume that $\sigma _{n}>0$ for all $n\geq 0$. Since $S
$ is a Kannan mapping, it follows from (\ref{4}) that%
\begin{eqnarray*}
\sigma _{n} &=&\rho (u_{n},u_{n+1})=\rho (Su_{n-1},Su_{n}) \\
&\leq &\lambda \left[ \rho (u_{n-1},Su_{n-1})+\rho (u_{n},Su_{n})\right]  \\
&=&\lambda \left[ \rho (u_{n-1},u_{n})+\rho (u_{n},u_{n+1})\right]  \\
&=&\lambda \left[ \sigma _{n-1}+\sigma _{n}\right] \text{.}
\end{eqnarray*}%
Therefore, we get $\sigma _{n}\leq \frac{\lambda }{1-\lambda }\sigma _{n-1}$%
. On repeating this process we obtain%
\begin{equation*}
\sigma _{n}\leq \left( \frac{\lambda }{1-\lambda }\right) ^{n}\sigma _{0}%
\text{.}
\end{equation*}%
From hypothesis, since $\lambda \in \lbrack 0,\frac{1}{2})$, we have%
\begin{equation}
\lim_{n\rightarrow \infty }\sigma _{n}=\lim_{n\rightarrow \infty }\rho
(u_{n},u_{n+1})=0.  \label{i}
\end{equation}
So, for every $\varepsilon >0$, there exists a natural number $n_{0}$ such
that $\sigma _{n}<\varepsilon /2$ and $\sigma _{m}<\varepsilon /2$ for all $%
n,m\geq n_{0}$. From (\ref{(i)}), we have%
\begin{eqnarray*}
\rho (u_{n},u_{m}) &=&\rho (Su_{n-1},Su_{m-1}) \\
&\leq &\lambda \left[ \rho (u_{n-1},Su_{n-1})+\rho (u_{m-1},Su_{m-1})\right] 
\\
&=&\lambda \left[ \rho (u_{n-1},u_{n})+\rho (u_{m-1},u_{m})\right]  \\
&=&\lambda \left[ \sigma _{n-1}+\sigma _{m-1}\right]  \\
&<&\frac{\varepsilon }{2}+\frac{\varepsilon }{2}=\varepsilon 
\end{eqnarray*}%
for $n,m>n_{0}$. Hence,\ $\left\{ u_{n}\right\} $ is Cauchy sequence in $E$
and $\lim_{n,m\rightarrow \infty }\rho (u_{n},u_{m})=0$. It follows from the
completeness of $E$ that there exists $b\in E$ such that%
\begin{equation*}
\lim_{n\rightarrow \infty }\rho (u_{n},b)=\lim_{n,m\rightarrow \infty }\rho
(u_{n},u_{m})=\rho (b,b)=0\text{.}
\end{equation*}%
Now,we show that $b$ is a fixed point of $S$. From definition of Kannan
mappings and partial $b_{v}\left( s\right) $ metric space, we have%
\begin{eqnarray*}
\rho (b,Sb) &\leq &s\left[ \rho (b,u_{n+1})+\rho (u_{n+1},u_{n+2})+\ldots
+\rho (u_{n+v-1},u_{n+v})+\right.  \\
&&\left. \rho (u_{n+v},Sb)\right] -\sum\limits_{i=1}^{v}\rho
(u_{n+i},u_{n+i}) \\
&\leq &s\left[ \rho (b,u_{n+1})+\rho (u_{n+1},u_{n+2})+\ldots +\rho
(u_{n+v-1},u_{n+v})+\rho (u_{n+v},Sb)\right]  \\
&\leq &s\left[ \rho (b,u_{n+1})+\rho (u_{n+1},u_{n+2})+\ldots +\rho
(u_{n+v-1},u_{n+v})+\rho (Su_{n+v-1},Sb)\right]  \\
&\leq &s\left[ \rho (b,u_{n+1})+\rho (u_{n+1},u_{n+2})+\ldots +\rho
(u_{n+v-1},u_{n+v})+\right.  \\
&&\left. \lambda \left\{ \rho (u_{n+v-1},Su_{n+v-1})+\rho (b,Sb)\right\} 
\right] .
\end{eqnarray*}%
So, it follows from the last inequality that%
\begin{eqnarray*}
\rho (b,Sb) &\leq &\frac{s}{(1-s\lambda )}\left[ \rho (b,u_{n+1})+\rho
(u_{n+1},u_{n+2})\right.  \\
&&\left. +...+\rho (u_{n+v-1},u_{n+v})+\lambda \rho (u_{n+v-1},Su_{n+v-1})
\right] \text{.}
\end{eqnarray*}%
Since $\lambda \neq \frac{1}{s}$ and $\left\{ u_{n}\right\} $ is a Cauchy
and convergent sequence, we have $\rho (b,Sb)=0$, so $Sb=b$. It means that $b
$ is a fixed point of $S$.

Now we show the uniqueness of fixed point. But first, we need to show that
if $b\in E$ is a fixed point of $S$, then $\rho (b,b)=0$. Let assume to the
contrary that $\rho (b,b)>0$. Then, from (\ref{4}) we have%
\begin{equation*}
\rho (b,b)=\rho (Sb,Sb)\leq \lambda \left[ \rho (b,Sb)+\rho (b,Sb)\right]
=2\lambda \rho (b,b)<\rho (b,b)\text{,}
\end{equation*}%
which is a contradiction. So, assumption is wrong, namely, $\rho (b,b)=0$.
Now, we can show that $S$ has a unique fixed point. Suppose $a,b\in E$ be
two distinct fixed points of $S$. Then we have $\rho (b,b)=\rho (a,a)=0$,
and it follows from (\ref{4}) that%
\begin{eqnarray*}
\rho (b,a) &=&\rho (Sb,Sa)\leq \lambda \left[ \rho (b,Sb)+\rho (a,Sa)\right] 
\\
&=&\lambda \left[ \rho (b,b)+\rho (a,a)\right] =0
\end{eqnarray*}%
Therefore, we have $\rho (b,a)=0$ and so $b=a$. Thus $S$ has a unique fixed
point. This completes the proof.
\end{proof}

\begin{theorem}
\label{C}Let $(E,\rho )$ be a complete partial $b_{v}\left( s\right) $
metric space and $S:E\rightarrow E$ a mapping satisfying:%
\begin{equation}
\rho (Su,Sw)\leq \lambda \max \left\{ \rho (u,w),\rho (u,Su),\rho
(w,Sw)\right\}   \label{6}
\end{equation}
for all $u,w\in E$ and $\lambda \in \left[ 0,\frac{1}{s}\right) $. Then, $S$
has a unique fixed point $b\in E$ and $\rho (b,b)=0$.
\end{theorem}

\begin{proof}
We begin with the existence of fixed points of $S$. Let $u_{0}\in E$ be an
arbitrary initial point and let $\left\{ u_{n}\right\} $ be a sequence
defined by $u_{n+1}=Su_{n}$ for all $n$.{\LARGE \ }If $u_{n}=u_{n+1}$ for at
least one natural number $n$, then it is clear that this point is a fixed
point of $S$. So, let assume that $u_{n+1}\neq u_{n}$ for all $n$. Now, it
follows from (\ref{6}) that%
\begin{eqnarray*}
\rho (u_{n+1},u_{n}) &=&\rho (Su_{n},Su_{n-1}) \\
&\leq &\lambda \max \left\{ \rho (u_{n},u_{n-1}),\rho (u_{n},Su_{n}),\rho
(u_{n-1},Su_{n-1})\right\}  \\
&=&\lambda \max \left\{ \rho (u_{n},u_{n-1}),\rho (u_{n},u_{n+1}),\rho
(u_{n-1},u_{n})\right\}  \\
&=&\lambda \max \left\{ \rho (u_{n},u_{n-1}),\rho (u_{n},u_{n+1})\right\} 
\text{.}
\end{eqnarray*}%
Set $L=\max \left\{ \rho (u_{n},u_{n-1}),\rho (u_{n},u_{n+1})\right\} $.
There exist two cases. If $L=\rho (u_{n},u_{n+1})$, then we get $\rho
(u_{n+1},u_{n})\leq \lambda \rho (u_{n+1},u_{n})<\rho (u_{n+1},u_{n})$ which
is a contradiction. So, we must have $L=\rho (u_{n},u_{n-1})$ and then we
have%
\begin{equation*}
\rho (u_{n+1},u_{n})\leq \lambda \rho (u_{n},u_{n-1})\text{.}
\end{equation*}%
On repeating this process, we obtain%
\begin{equation}
\rho (u_{n+1},u_{n})\leq \lambda ^{n}\rho (u_{1},u_{0})  \label{7}
\end{equation}
for all $n$. On the other hand, since $\lambda ^{n}\rightarrow 0$ for $%
n\rightarrow \infty $, there exists a natural number $n_{0}$ such that $%
0<\lambda ^{n_{0}}s<1$. For $m,n\in 
\mathbb{N}
$ with $m>n$, by using inequality (\ref{7}), we obtain%
\begin{eqnarray*}
\rho (u_{n},u_{m}) &\leq &s\left[ \rho (u_{n},u_{n+1})+\rho
(u_{n+1},u_{n+2})+\ldots +\rho (u_{n+v-3},u_{n+v-2})\right.  \\
&&\left. +\rho (u_{n+v-2},u_{n+n_{0}})+\rho (u_{n+n_{0}},u_{m+n_{0}})+\rho
(u_{m+n_{0}},u_{m})\right]  \\
&&-\sum\limits_{i=1}^{v-2}\rho (u_{n+i},u_{n+i})-\rho
(u_{n+n_{0}},u_{n+n_{0}})-\rho (u_{m+n_{0}},u_{m+n_{0}}) \\
&\leq &s\left[ \rho (u_{n},u_{n+1})+\rho (u_{n+1},u_{n+2})+\ldots +\rho
(u_{n+v-3},u_{n+v-2})\right.  \\
&&\left. +\rho (u_{n+v-2},u_{n+n_{0}})+\rho (u_{n+n_{0}},u_{m+n_{0}})+\rho
(u_{m+n_{0}},u_{m})\right]  \\
&\leq &s\left( \lambda ^{n}+\lambda ^{n+1}+\cdots +\lambda ^{n+v-3}\right)
\rho (u_{0},u_{1}) \\
&&+s\lambda ^{n}\rho (u_{v-2},u_{n_{0}})+s\lambda ^{n_{0}}\rho
(u_{n},u_{m})+s\lambda ^{m}\rho (u_{n_{0}},u_{0}).
\end{eqnarray*}%
So, we get%
\begin{eqnarray*}
\left( 1-s\lambda ^{n_{0}}\right) \rho (u_{n},u_{m}) &\leq &s\left( \lambda
^{n}+\lambda ^{n+1}+\cdots +\lambda ^{n+v-3}\right) \rho (u_{0},u_{1}) \\
&&+s\lambda ^{n}\rho (u_{v-2},u_{n_{0}})+s\lambda ^{m}\rho (u_{n_{0}},u_{0}).
\end{eqnarray*}%
By taking limit from both side, we have%
\begin{equation*}
\lim_{n,m\rightarrow \infty }\rho (u_{n},u_{m})=0
\end{equation*}

Therefore, $\left\{ u_{n}\right\} $ is a Cauchy sequence in $E$. By
completeness of $E,$ there exists $b\in E$ such that%
\begin{equation}
\lim_{n\rightarrow \infty }\rho (u_{n},b)=\lim_{n,m\rightarrow \infty }\rho
(u_{n},u_{m})=\rho (b,b)=0\text{.}  \label{8}
\end{equation}%
Now, we show that $b$ is a fixed point of $S$. From definition of partial $%
b_{v}\left( s\right) $ metric space and inequality (\ref{6}), we have%
\begin{eqnarray*}
\rho (b,Sb) &\leq &s\left[ \rho (b,u_{n+1})+\rho (u_{n+1},u_{n+2})+\ldots
+\rho (u_{n+v-1},u_{n+v})+\right.  \\
&&\left. \rho (u_{n+v},Sb)\right] -\sum\limits_{i=1}^{v}\rho
(u_{n+i},u_{n+i}) \\
&\leq &s\left[ \rho (b,u_{n+1})+\rho (u_{n+1},u_{n+2})+\ldots +\rho
(u_{n+v-1},u_{n+v})+\rho (u_{n+v},Sb)\right]  \\
&\leq &s\left[ \rho (b,u_{n+1})+\rho (u_{n+1},u_{n+2})+\ldots +\rho
(u_{n+v-1},u_{n+v})+\rho (Su_{n+v-1},Sb)\right]  \\
&\leq &s\left[ \rho (b,u_{n+1})+\rho (u_{n+1},u_{n+2})+\ldots +\rho
(u_{n+v-1},u_{n+v})+\right.  \\
&&\left. \lambda \max \left\{ \rho (u_{n+v-1},b),\rho
(u_{n+v-1},u_{n+v}),\rho (b,Sb)\right\} \right] .
\end{eqnarray*}%
Set $F=\max \left\{ \rho (u_{n+v-1},b),\rho (u_{n+v-1},u_{n+v}),\rho
(b,Sb)\right\} $. There exists three cases:

1. If $F=\rho (u_{n+v-1},b)$, then we get%
\begin{equation*}
\rho (b,Sb)\leq s\left[ \rho (b,u_{n+1})+\rho (u_{n+1},u_{n+2})+\ldots +\rho
(u_{n+v-1},u_{n+v})+\lambda \rho (u_{n+v-1},b)\right] .
\end{equation*}%
So, it follows from (\ref{8}) that $\rho (b,Sb)=0$.

2. If $F=\rho (u_{n+v-1},u_{n+v})$, then we get%
\begin{equation*}
\rho (b,Sb)\leq s\left[ \rho (b,u_{n+1})+\rho (u_{n+1},u_{n+2})+\ldots
+\left( 1+\lambda \right) \rho (u_{n+v-1},u_{n+v})\right] .
\end{equation*}%
Again by using (\ref{8}), we obtain that $\rho (b,Sb)=0$.

3. If $F=\rho (b,Sb)$ then we get%
\begin{equation*}
\left( 1-s\lambda \right) \rho (b,Sb)\leq s\left[ \rho (b,u_{n+1})+\rho
(u_{n+1},u_{n+2})+\ldots +\rho (u_{n+v-1},u_{n+v})\right] .
\end{equation*}%
Since $\lambda \in \left[ 0,\frac{1}{s}\right) $, we obtain that $\rho
(b,Sb)=0$, that is, $Sb=b$. Thus, $b$ is a fixed poit of $S$. 

Now we show the uniqueness of fixed point of $S$. Suppose on the contrary
that $a$ and $b$ are two distinct fixed points of $S$ and $\rho (a,b)>0$. It
follows from (\ref{6}) that%
\begin{eqnarray*}
\rho (a,b) &=&\rho (Sa,Sb)\leq \lambda \max \left\{ \rho (a,b),\rho
(a,Sa),\rho (b,Sb)\right\}  \\
&=&\lambda \max \left\{ \rho (a,b),\rho (a,a),\rho (b,b)\right\}  \\
&=&\lambda \rho (a,b)<\rho (a,b)\text{,}
\end{eqnarray*}%
which is a cotradiction. Therefore, we must have $\rho (a,b)=0$ and so $a=b$%
. Hence, $S$ has a unique fixed point.
\end{proof}

In definition \ref{Def}, if we take $s=1$, then we derive following
definition of partial $v$-generalized metric space.

\begin{definition}
Let $E$ be a nonempty set and $\rho :$ $E\times E\rightarrow \lbrack
0,\infty )$ be a mapping and $v\in 
\mathbb{N}
$. Then $(E,\rho )$ is said to be a partial $v$-generalized metric space if
following conditions hold for all $u,w,z_{1},z_{2},\ldots ,z_{v}\in E$:
\end{definition}

\begin{enumerate}
\item $u=w\Leftrightarrow \rho (u,u)=\rho (u,w)=\rho (w,w)$;

\item $\rho (u,u)\leq \rho (u,w)$;

\item $\rho (u,w)=\rho (w,u)$;

\item $\rho (u,w)\leq \rho (u,z_{1})+\rho (z_{1},z_{2})+\ldots +\rho
(z_{v-1},z_{v})+\rho (z_{v},y)-\sum_{i=1}^{v}\rho (z_{i},z_{i})$.
\end{enumerate}

In Theorems \ref{A},\ref{B} and \ref{C}, if take $s=1$, then we derive
following fixed point theorems in partial $v$-generalized metric space.

\begin{corollary}
Let $(E,\rho )$ be a complete partial $v$-generalized metric space and $%
S:E\rightarrow E$ be a contraction mapping, i.e., $S$ satisfies%
\begin{equation*}
\rho (Su,Sw)\leq \lambda \rho (u,w)
\end{equation*}%
for all $u,w\in E,$ where $\lambda \in \lbrack 0,1)$. Then $S$ has a unique
fixed point $b\in S$ and $\rho (b,b)=0$.
\end{corollary}

\begin{corollary}
Let $(E,\rho )$ be a complete partial $v$-generalized metric space and $%
S:E\rightarrow E$ a mapping satisfying the following condition:%
\begin{equation*}
\rho (Su,Sy)\leq \lambda \left[ \rho (u,Su)+\rho (w,Sw)\right] 
\end{equation*}%
for all $u,w\in E$, where $\lambda \in \lbrack 0,\frac{1}{2})$. Then $S$ has
a unique fixed point $b\in E$ and $\rho (b,b)=0$.
\end{corollary}

\begin{corollary}
Let $(E,\rho )$ be a complete partial $v$-generalized metric space and $%
S:E\rightarrow E$ a mapping satisfying:%
\begin{equation*}
\rho (Su,Sw)\leq \lambda \max \left\{ \rho (u,w),\rho (u,Su),\rho
(w,Sw)\right\} 
\end{equation*}%
for all $u,w\in E$ and $\lambda \in \left[ 0,1\right) $. Then, $S$ has a
unique fixed point $b\in E$ and $\rho (b,b)=0$.
\end{corollary}

\subsection{$b_{v}\left( \protect\theta \right) $ Metric Spaces}

In 2017, Kamran et al. introduced following generalized metric space which
they call extended $b$-metric space.

\begin{definition}
\cite{tay} Let $E$ be a nonempty set and let $\theta :E\times E\rightarrow
\lbrack 1,\infty )$ be a function. A function $\rho _{\theta }:E\times
E\rightarrow \lbrack 0,\infty )$ is called an extended $b$-metric if for all 
$u,v,w\in E$ it satisfies:

\begin{enumerate}
\item[(1)] $\rho _{\theta }(u,w)=0$ iff $u=w$;

\item[(2)] $\rho _{\theta }(u,w)=\rho (w,u)$;

\item[(3)] $\rho _{\theta }(u,w)\leq \theta \left( u,w\right) [\rho _{\theta
}(u,v)+\rho _{\theta }(v,w)].$

The pair $\left( E,\rho _{\theta }\right) $ is called an extended $b$-metric
space.
\end{enumerate}
\end{definition}

It is clear that if $\theta \left( u,w\right) =s$ for all $u,w\in E,$ then
we obtain $b$-metric space.

From this point of view, we introduce following generalized metric space
called as $b_{v}(\theta )$ (or extended $b_{v}(s)$ ) metric space.

\begin{definition}
Let $E$ be a nonempty set, $\theta :E\times E\rightarrow \lbrack 1,\infty )$
a function and $v\in 
\mathbb{N}
$. Then $\rho _{\theta }:E\times E\rightarrow \lbrack 0,\infty )$ is called $%
b_{v}(\theta )$ metric if for all $u,z_{1},z_{2},...,z_{v},w\in E$, each of
them different from each other, it satisfies

\begin{enumerate}
\item[(1)] $\rho _{\theta }(u,w)=0$ iff $u=w$;

\item[(2)] $\rho _{\theta }(u,w)=\rho _{\theta }(w,u)$;

\item[(3)] $\rho _{\theta }(u,w)\leq \theta \left( u,w\right) [\rho _{\theta
}(u,z_{1})+\rho _{\theta }(z_{1},z_{2})+\cdots +\rho _{\theta }\left(
z_{v},w\right) ].$

The pair $\left( E,\rho _{\theta }\right) $ is called $b_{v}(\theta )\ $%
metric space.
\end{enumerate}
\end{definition}

\begin{remark}
It is clear that if for all $u,w\in E$

\begin{enumerate}
\item $\theta \left( u,w\right) =s$, then we obtain $b_{v}(s)$ metric space,

\item $v=1$, then we obtain extended $b$-metric space,

\item $\theta \left( u,w\right) =s$ and $v=1$, then we obtain $b$-metric
space,

\item $\theta \left( u,w\right) =s$ and $v=2$, then we obtain rectangular $b$%
-metric space,

\item $\theta \left( u,w\right) =1$ and $v=2$, then we obtain rectangular
metric space,

\item $\theta \left( u,w\right) =1$, then we obtain $v$-generalized metric
space,

\item $\theta \left( u,w\right) =1$ and $v=1$, then we obtain usual metric
space.
\end{enumerate}
\end{remark}

\begin{example}
Let $E=%
\mathbb{N}
$. Define mappings $\theta :%
\mathbb{N}
\times 
\mathbb{N}
\rightarrow \lbrack 1,\infty )$ and $\rho _{\theta }:%
\mathbb{N}
\times 
\mathbb{N}
\rightarrow \lbrack 0,\infty )$ by $\theta \left( u,w\right) =3+u+w$ and 
\begin{equation*}
\rho _{\theta }\left( u,w\right) =\left\{ 
\begin{array}{ll}
6, & \text{if }u,w\in \left\{ 1,2\right\} \text{ and }u\neq w \\ 
1, & \text{if }u\text{ or }w\notin \left\{ 1,2\right\} \text{ and }u\neq w
\\ 
0, & \text{if }u=w%
\end{array}%
\right.
\end{equation*}%
for all $u,w\in 
\mathbb{N}
$. Then, it is easy to see that $\left( E,\rho _{\theta }\right) $ is a $%
b_{v}\left( \theta \right) $ metric space with $v=5$.
\end{example}

Definitions of Cauchy sequence, convergence and completeness can be easily
extended to the case of $b_{v}\left( \theta \right) $ metric space by the
following way.

\begin{definition}
Let $\left( E,\rho _{\theta }\right) $ be a $b_{v}\left( \theta \right) $
metric space, $\left\{ u_{n}\right\} $ a sequence in $E$ and $u\in E$. Then,

\begin{enumerate}
\item[a)] $\left\{ u_{n}\right\} $ is said to converge to $u$ in $\left(
E,\rho _{\theta }\right) $ if for every $\varepsilon >0$, there exists $%
n_{0}\in 
\mathbb{N}
$ such that $\rho _{\theta }\left( u_{n},u\right) <\varepsilon $ for all $%
n\geq n_{0}$ and this convergence is denoted by $u_{n}\rightarrow u$.

\item[b)] $\left\{ u_{n}\right\} $ is said to be Cauchy sequence in $\left(
E,\rho _{\theta }\right) $ if for every $\varepsilon >0$, there exists $%
n_{0}\in 
\mathbb{N}
$ such that $\rho _{\theta }\left( u_{n},u_{n+p}\right) <\varepsilon $ for
all $n\geq n_{0}$ and $p>0.$

\item[c)] $\left( E,\rho _{\theta }\right) $ is said to be complete if every
Cauchy sequence in $E$ is convergent in $E$.
\end{enumerate}
\end{definition}

Now, we are in the position to prove fixed point theorems in $b_{v}(\theta )$
metric spaces. But first, we prove following lemmas which we need in the
proof of main theorems.

\begin{lemma}
\label{A1}Let $\left( E,\rho _{\theta }\right) $ be a $b_{v}(\theta )$
metric space, $S:E\rightarrow E$ a mapping and $\left\{ u_{n}\right\} $ a
sequence in $E$ defined by $u_{n+1}=Su_{n}=S^{n}u_{0}$ such that $u_{n}\neq
u_{n+1}$. Suppose that $c\in \lbrack 0,1)$ such that%
\begin{equation*}
\rho _{\theta }\left( u_{n+1},u_{n}\right) \leq c\rho _{\theta }\left(
u_{n},u_{n-1}\right)
\end{equation*}%
for all $n\in 
\mathbb{N}
$. Then $u_{n}\neq u_{m}$ for all distinct $n,m\in 
\mathbb{N}
$.
\end{lemma}

\begin{proof}
Since the proof is very similar with the proof of Lemma 1.11 of \cite{mr},
we omit it.
\end{proof}

\begin{lemma}
\label{A2}Let $\left( E,\rho _{\theta }\right) $ be a $b_{v}(\theta )$
metric space with a bounded function $\theta $ and $\left\{ u_{n}\right\} $
a sequence in $E$ defined by $u_{n+1}=Su_{n}=S^{n}u_{0}$ such that $%
u_{n}\neq u_{m}$ for all $n,m\in 
\mathbb{N}
$. Assume that there exist $c\in \lbrack 0,1)$ and $k_{1},k_{2}\in 
\mathbb{R}
^{+}\cup \left\{ 0\right\} $ such that%
\begin{equation}
\rho _{\theta }\left( u_{m},u_{n}\right) \leq c\rho _{\theta }\left(
u_{m-1},u_{n-1}\right) +k_{1}c^{m}+k_{2}c^{m}  \label{2.1}
\end{equation}%
for all $n,m\in 
\mathbb{N}
$. Then $\left\{ u_{n}\right\} $ is a Cauchy sequence in $E$.
\end{lemma}

\begin{proof}
It is easy to see that $\left\{ u_{n}\right\} $ is Cauchy if $c=0$. So, we
should assume that $c\neq 0$. Since function $\theta \left( u,w\right) $ is
bounded, there exists a number $n_{0}\in 
\mathbb{N}
$ such that%
\begin{equation}
0<c^{n_{0}}\theta \left( u,w\right) <1  \label{2.2}
\end{equation}%
for all $u,w\in E$. From hypothesis of lemma, we can write%
\begin{eqnarray*}
\rho _{\theta }\left( u_{n+1},u_{n}\right) &\leq &c\rho _{\theta }\left(
u_{n},u_{n-1}\right) +k_{1}c^{n+1}+k_{2}c^{n} \\
&\leq &c\left( c\rho _{\theta }\left( u_{n-1},u_{n-2}\right)
+k_{1}c^{n}+k_{2}c^{n-1}\right) +k_{1}c^{n+1}+k_{2}c^{n} \\
&=&c^{2}\rho _{\theta }\left( u_{n-1},u_{n-2}\right) +2\left(
k_{1}c^{n+1}+k_{2}c^{n}\right) \\
&&\vdots \\
&\leq &c^{n}\rho _{\theta }\left( u_{1},u_{0}\right) +n\left(
k_{1}c^{n+1}+k_{2}c^{n}\right) .
\end{eqnarray*}%
Similarly, for all $k\geq 1$, we can write%
\begin{equation*}
\rho _{\theta }\left( u_{m+k},u_{n+k}\right) \leq c^{k}\rho _{\theta }\left(
u_{m},u_{n}\right) +k\left( k_{1}c^{m+k}+k_{2}c^{n+k}\right) .
\end{equation*}%
If $v\geq 2$, then from the definition of $b_{v}(\theta )$ metric space, we
get%
\begin{eqnarray*}
\rho _{\theta }\left( u_{n},u_{m}\right) &\leq &\theta \left(
u_{n},u_{m}\right) \left[ \rho _{\theta }\left( u_{n},u_{n+1}\right) +\rho
_{\theta }\left( u_{n+1},u_{n+2}\right) \right. \\
&&+\cdots +\rho _{\theta }\left( u_{n+v-3},u_{n+v-2}\right) +\rho _{\theta
}\left( u_{n+v-2},u_{n+n_{0}}\right) \\
&&\left. +\rho _{\theta }\left( u_{n+n_{0}},u_{m+n_{0}}\right) +\rho
_{\theta }\left( u_{m+n_{0}},u_{m}\right) \right] .
\end{eqnarray*}%
Then, we have%
\begin{eqnarray*}
\rho _{\theta }\left( u_{n},u_{m}\right) &\leq &\theta \left(
u_{n},u_{m}\right) \left[ \left( c^{n}+c^{n+1}+\cdots +c^{n+v-3}\right) \rho
_{\theta }\left( u_{0},u_{1}\right) \right. \\
&&+\left( k_{1}c+k_{2}\right) \left( nc^{n}+\left( n+1\right) c^{n+1}+\cdots
+\left( n+v-3\right) c^{n+v-3}\right) \\
&&+c^{n}\rho _{\theta }\left( u_{v-2},u_{n_{0}}\right) +nc^{n}\left(
k_{1}c^{v-2}+k_{2}c^{n_{0}}\right) \\
&&+c^{n_{0}}\rho _{\theta }\left( u_{n},u_{m}\right) +n_{0}c^{n_{0}}\left(
k_{1}c^{n}+k_{2}c^{m}\right) \\
&&+\left. c^{m}\rho _{\theta }\left( u_{n_{0}},u_{0}\right) +mc^{m}\left(
k_{1}c^{n_{0}}+k_{2}\right) \right] .
\end{eqnarray*}%
So, we obtain%
\begin{eqnarray*}
\rho _{\theta }\left( u_{n},u_{m}\right) \left( 1-c^{n_{0}}\theta \left(
u_{n},u_{m}\right) \right) &\leq &\theta \left( u_{n},u_{m}\right) \left[
\left( c^{n}+c^{n+1}+\cdots +c^{n+v-3}\right) \rho _{\theta }\left(
u_{0},u_{1}\right) \right. \\
&&+\left( k_{1}c+k_{2}\right) \left( nc^{n}+\left( n+1\right) c^{n+1}+\cdots
+\left( n+v-3\right) c^{n+v-3}\right) \\
&&+c^{n}\rho _{\theta }\left( u_{v-2},u_{n_{0}}\right) +nc^{n}\left(
k_{1}c^{v-2}+k_{2}c^{n_{0}}\right) +n_{0}c^{n_{0}}\left(
k_{1}c^{n}+k_{2}c^{m}\right) \\
&&+\left. c^{m}\rho _{\theta }\left( u_{n_{0}},u_{0}\right) +mc^{m}\left(
k_{1}c^{n_{0}}+k_{2}\right) \right] .
\end{eqnarray*}%
Since $\lim_{n\rightarrow \infty }nc^{n}=0$ and $1-c^{n_{0}}\theta \left(
u_{n},u_{m}\right) >0$, using (\ref{2.1}), we have $\rho _{\theta }\left(
u_{n},u_{m}\right) \rightarrow 0$ as $n,m\rightarrow \infty $. This means
that $\left\{ u_{n}\right\} $ is a Cauchy sequence. Since $b_{v}\left(
s\right) $ metric space is a $b_{2v}\left( s^{2}\right) $ metric space, if $%
v=1$, then $\left\{ u_{n}\right\} $ is Cauchy.
\end{proof}

Now we can give Banach fixed point theorem in complete $b_{v}\left( \theta
\right) $ metric space.

\begin{theorem}
\label{A3}Let $\left( E,\rho _{\theta }\right) $ be a complete $b_{v}(\theta
)$ metric space with a bounded function $\theta $ and $S:E\rightarrow E$ a
contraction mapping, i.e., there exists a constant $c\in \left[ 0,1\right) $
such that 
\begin{equation}
\rho _{\theta }\left( Su,Sw\right) \leq c\rho _{\theta }\left( u,w\right)
\label{2.3}
\end{equation}%
for all $u,w\in E$. Then $S$ has a unique fixed point.
\end{theorem}

\begin{proof}
Let $u_{0}\in E$ be an arbitrary initial point and let $\left\{
u_{n}\right\} $ be a sequence defined by $u_{n+1}=Su_{n}=S^{n+1}u_{0}$ and $%
u_{n}\neq u_{n+1}$ for all $n\geq 0$. It follows from Lemma \ref{A1} that $%
u_{n}\neq u_{m}$ for all $n,m\in 
\mathbb{N}
$. Since $S$ is a contraction mapping, we can write%
\begin{equation*}
\rho _{\theta }\left( u_{n},u_{m}\right) =\rho _{\theta }\left(
Su_{n-1},Su_{m-1}\right) \leq c\rho _{\theta }\left( u_{n-1},u_{m-1}\right) .
\end{equation*}%
From Lemma \ref{A2}, we have $\left\{ u_{n}\right\} $ is a Cauchy sequence.
So, it follows from completeness of $E$ that there exists an element $u\in E$
such that $u_{n}\rightarrow u$. Now, we show that $u\in FixS$, i.e., $u=Su$. 
\begin{eqnarray*}
\rho _{\theta }\left( u,Su\right) &\leq &\theta \left( u,Su\right) \left[
\rho _{\theta }\left( u,u_{n+1}\right) +\rho _{\theta }\left(
u_{n+1},u_{n+2}\right) \right. \\
&&\left. +\cdots +\rho _{\theta }\left( u_{n+v-1},u_{n+v}\right) +\rho
_{\theta }\left( u_{n+v},Su\right) \right] \\
&=&\theta \left( u,Su\right) \left[ \rho _{\theta }\left( u,u_{n+1}\right)
+\rho _{\theta }\left( u_{n+1},u_{n+2}\right) \right. \\
&&\left. +\cdots +\rho _{\theta }\left( u_{n+v-1},u_{n+v}\right) +\rho
_{\theta }\left( Su_{n+v-1},Su\right) \right] \\
&\leq &\theta \left( u,Su\right) \left[ \rho _{\theta }\left(
u,u_{n+1}\right) +\rho _{\theta }\left( u_{n+1},u_{n+2}\right) \right. \\
&&\left. +\cdots +\rho _{\theta }\left( u_{n+v-1},u_{n+v}\right) +c\rho
_{\theta }\left( u_{n+v-1},u\right) \right] .
\end{eqnarray*}%
Since $\theta $ is a bounded function and $\left\{ u_{n}\right\} $ is Cauchy
with $u_{n}\rightarrow u$, we have $\rho _{\theta }\left( u,Su\right) =0$.
This means that $u\in FixS$. Next, we need to show that $u$ is a unique
fixed point. Let assume to the contrary that there exists another fixed
point $w$. Since 
\begin{equation*}
\rho _{\theta }\left( u,w\right) =\rho _{\theta }\left( Su,Sw\right) \leq
c\rho _{\theta }\left( u,w\right) <\rho _{\theta }\left( u,w\right) ,
\end{equation*}%
we get $u=w$ that is $u$ is the unique fixed point of $S$.
\end{proof}

\begin{remark}
In Theorem \ref{A3},

\begin{enumerate}
\item if we take the constant $v=1$ and the function $\theta \left(
u,w\right) =1$ for all $u,w\in E$, then we derive classical Banach fixed
point theorem in usual metric spaces.

\item if we take $\theta \left( u,w\right) =s$ for all $u,w\in E$ where $%
s\geq 1$, then we derive Theorem 2.1 of \cite{mr} in $b_{v}\left( s\right) $
metric spaces.

\item if $v=1$ and $\theta \left( u,w\right) =s$ for all $u,w\in E$, then we
derive Theorem 2.1 of \cite{dun} in $b$-metric spaces.

\item if $v=2$ and $\theta \left( u,w\right) =s$ for all $u,w\in E$, then we
derive Theorem 2.1 of \cite{mit} and so main theorem of \cite{geo} in
rectangular $b$-metric spaces.

\item if $\theta \left( u,w\right) =1$ for all $u,w\in E$, then we derive
main result of Branciari \cite{bri} in $v$-generalized metric spaces.
\end{enumerate}
\end{remark}

In literature, there exist various type of contraction mappings. Weakly
contractive mapping is one of this type of contractions which generalize
usual contractions. A mapping $S:E\rightarrow E$ is called weakly
contractive if there exists a continuous and nondecreasing function $\psi
\left( t\right) $ defined from $%
\mathbb{R}
^{+}\cup \left\{ 0\right\} $ onto itself such that $\psi \left( 0\right)
=0,\psi \left( t\right) \rightarrow \infty $ as $t\rightarrow \infty $ and
for all $u,w\in E$%
\begin{equation}
\rho _{\theta }\left( Su,Sw\right) \leq \rho _{\theta }\left( u,w\right)
-\psi \left( \rho _{\theta }\left( u,w\right) \right) .  \label{m2}
\end{equation}

Now, we generalize Banach fixed point theorem for weakly contractive
mappings in $b_{v}\left( \theta \right) $ metric space.

\begin{theorem}
\label{AA}Let $E$ be a complete $b_{v}\left( \theta \right) $ metric space
and $S$ a weakly contractive mapping on $E$. Then $S$ has a unique fixed
point.
\end{theorem}

\begin{proof}
Let $u_{0}\in E$ be an arbitrary initial point. Define sequence $\left\{
u_{n}\right\} $ by $u_{1}=Su_{0}$, $u_{2}=Su_{1}=S^{2}u_{0},\ldots
,u_{n+1}=Su_{n}=S^{n}u_{0}$. If $u_{n}=u_{n+1}$ for all $n\in 
\mathbb{N}
$ where $%
\mathbb{N}
$ is the set of positive integer, then proof is trivial. So, let assume that 
$u_{n}\neq u_{n+1}$ for all $n$. Moreover, the case that $u_{n}\neq u_{m}$
for all different $n$ and $m$ can be easily proved. From (\ref{m2}), we can
write%
\begin{eqnarray*}
\rho _{\theta }\left( u_{n+1},u_{n+p+1}\right) &=&\rho _{\theta }\left(
Su_{n},Su_{n+p}\right) \\
&\leq &\rho _{\theta }\left( u_{n},u_{n+p}\right) -\psi \left( \rho _{\theta
}\left( u_{n},u_{n+p}\right) \right)
\end{eqnarray*}%
for all $n,p\in 
\mathbb{N}
$. Let $\alpha _{n}=\rho _{\theta }\left( u_{n},u_{n+p}\right) $. Since $%
\psi $ is nondecreasing, we have%
\begin{equation}
\alpha _{n+1}\leq \alpha _{n}-\psi \left( \alpha _{n}\right) \leq \alpha
_{n}.  \label{m1}
\end{equation}%
Thus, the sequence $\left\{ \alpha _{n}\right\} $ has a limit $\alpha \geq 0$%
.\ Now we should show that $\alpha =0$. Assume to the contrary that $\alpha
>0$. Using (\ref{m1}), we have%
\begin{equation*}
\psi \left( \alpha _{n}\right) \geq \psi \left( \alpha \right) >0.
\end{equation*}%
So, we get%
\begin{equation*}
\alpha _{n+1}\leq \alpha _{n}-\psi \left( \alpha \right) .
\end{equation*}%
Hence, we obtain $\alpha _{N+m}\leq \alpha _{m}-N\psi \left( \alpha \right) $
which is a contradiction for large enough $N$. This proves that $\alpha =0$.
This means that $\left\{ u_{n}\right\} $ is Cauchy. Completeness of $E$
implies that there exists a point $u\in E$ such that $u_{n}\rightarrow u$.
Now, we show that $u$ is a fixed point of $S$. Using (\ref{m2}) and
definition of $\rho _{\theta }$, we get%
\begin{eqnarray*}
\rho _{\theta }\left( u,Su\right) &\leq &\theta \left( u,Su\right) \left[
\rho _{\theta }\left( u,u_{n+1}\right) +\rho _{\theta }\left(
u_{n+1},u_{n+2}\right) \right. \\
&&\left. +\ldots +\rho _{\theta }\left( u_{n+v-1},u_{n+v}\right) +\rho
_{\theta }\left( u_{n+v},Su\right) \right] \\
&=&\theta \left( u,Su\right) \left[ \rho _{\theta }\left( u,u_{n+1}\right)
+\rho _{\theta }\left( u_{n+1},u_{n+2}\right) \right. \\
&&\left. +\ldots +\rho _{\theta }\left( u_{n+v-1},u_{n+v}\right) +\rho
_{\theta }\left( Su_{n+v-1},Su\right) \right] \\
&\leq &\theta \left( u,Su\right) \left[ \rho _{\theta }\left(
u,u_{n+1}\right) +\rho _{\theta }\left( u_{n+1},u_{n+2}\right) +\ldots
+\right. \\
&&\left. \rho _{\theta }\left( u_{n+v-1},u_{n+v}\right) +\rho _{\theta
}\left( u_{n+v-1},u\right) -\psi \left( \rho _{\theta }\left(
u_{n+v-1},u\right) \right) \right] .
\end{eqnarray*}%
Since $\rho _{\theta }\left( u_{n},u_{n+p}\right) \rightarrow 0$ and $%
u_{n}\rightarrow u$ as $n\rightarrow \infty $ and $\psi \left( 0\right) =0$,
we have $u$ is a fixed point of $S$.

To prove the uniqueness of fixed point, we can assume that there exist one
more fixed point $w$. Since $S$ is a weakly contractive mapping, we have%
\begin{equation*}
\rho _{\theta }\left( u,w\right) =\rho _{\theta }\left( Su,Sw\right) \leq
\rho _{\theta }\left( u,w\right) -\psi \left( \rho _{\theta }\left(
u,w\right) \right) <\rho _{\theta }\left( u,w\right) .
\end{equation*}
So $u=w$ This finishes the proof.
\end{proof}

\begin{remark}
In Theorem \ref{AA},
\end{remark}

1 if we take the constant $v=1$, the function $\theta \left( u,w\right) =1$
for all $u,w\in E$ and $\psi (t)=ct$, then we derive classical Banach fixed
point theorem.

2.if we take $\psi (t)=ct$ and $\theta \left( u,w\right) =s$ where $s\in %
\left[ 1,\infty \right) $, then we derive Theorem 2.1 of \cite{mr}

3.if $v=1,$ $\theta \left( u,w\right) =s$ and $\psi (t)=ct$, then we derive
Theorem 2.1 of \cite{dun}.

4.if $v=2$, $\theta \left( u,w\right) =s$ and $\psi (t)=ct$, then we derive
Theorem 2.1 of \cite{mit} and so main theorem of \cite{geo}.

5.if $v=1$ and $\theta \left( u,w\right) =s$, then we derive main theorem of 
\cite{rho}.

Now, we give Reich fixed point theorem.

\begin{theorem}
\label{A4}Let $\left( E,\rho _{\theta }\right) $ be a complete $b_{v}(\theta
)$ metric space with a bounded function $\theta $ and $S:E\rightarrow E$ a
mapping satisfying:%
\begin{equation}
\rho _{\theta }\left( Su,Sw\right) \leq \alpha \rho _{\theta }\left(
u,w\right) +\beta \rho _{\theta }\left( u,Su\right) +\gamma \rho _{\theta
}\left( w,Sw\right)  \label{2.4}
\end{equation}%
for all $u,w\in E$ where $\alpha ,\beta ,\gamma $ are nonnegative constants
with $\alpha +\beta +\gamma <1$ and $\Gamma _{1}<\frac{1}{\Gamma _{2}}$
where $\Gamma _{1}=\min \left\{ \beta ,\gamma \right\} $ and $\Gamma
_{2}=\max \left\{ \theta \left( u,Su\right) ,\theta \left( Su,u\right)
\right\} $. Then $S$ has a unique fixed point. Moreover, sequence $\left\{
u_{n}\right\} $ defined by $u_{n}=Su_{n-1}$ converges strongly to the unique
fixed point of $S$.
\end{theorem}

\begin{proof}
Let $\left\{ u_{n}\right\} $ be a sequence defined by $%
u_{n+1}=Su_{n}=S^{n+1}u_{0}$ where $u_{0}\in E$ is an arbitrary initial
point. If $u_{n}=u_{n+1}$ for all $n\in 
\mathbb{N}
$, it is easy to see that $u_{0}$ is a fixed point of $S$. Now, we assume
that $u_{n}\neq u_{n+1}$ for all $n$. From (\ref{2.4}) and definition of $%
\left\{ u_{n}\right\} $, we have%
\begin{eqnarray*}
\rho _{\theta }\left( u_{n+1},u_{n}\right) &=&\rho _{\theta }\left(
Su_{n},Su_{n-1}\right) \\
&\leq &\alpha \rho _{\theta }\left( u_{n},u_{n-1}\right) +\beta \rho
_{\theta }\left( u_{n},Su_{n}\right) +\gamma \rho _{\theta }\left(
u_{n-1},Su_{n-1}\right) \\
&=&\alpha \rho _{\theta }\left( u_{n},u_{n-1}\right) +\beta \rho _{\theta
}\left( u_{n},u_{n+1}\right) +\gamma \rho _{\theta }\left(
u_{n-1},u_{n}\right) .
\end{eqnarray*}%
Then, we get%
\begin{eqnarray*}
\rho _{\theta }\left( u_{n+1},u_{n}\right) &\leq &\frac{\alpha +\gamma }{%
1-\beta }\rho _{\theta }\left( u_{n},u_{n-1}\right) \\
&\leq &\left( \frac{\alpha +\gamma }{1-\beta }\right) ^{n}\rho _{\theta
}\left( u_{1},u_{0}\right) .
\end{eqnarray*}%
Since $\alpha +\beta +\gamma <1$, then it is clear that $0\leq \frac{\alpha
+\gamma }{1-\beta }<1$. So, we obtain%
\begin{equation}
\lim_{n\rightarrow \infty }\rho _{\theta }\left( u_{n+1},u_{n}\right) =0.
\label{2.5}
\end{equation}%
Also, since we assume that $u_{n}\neq u_{n+1}$ for all $n$ and $\rho
_{\theta }\left( u_{n+1},u_{n}\right) \leq \frac{\alpha +\gamma }{1-\beta }%
\rho _{\theta }\left( u_{n},u_{n-1}\right) $, then it follows from Lemma \ref%
{A1} that $u_{n}\neq u_{m}$ for all $n,m\in 
\mathbb{N}
$. So, we have%
\begin{eqnarray*}
\rho _{\theta }\left( u_{n},u_{m}\right) &=&\rho _{\theta }\left(
Su_{n-1},Su_{m-1}\right) \\
&\leq &\alpha \rho _{\theta }\left( u_{n-1},u_{m-1}\right) +\beta \rho
_{\theta }\left( u_{n-1},Su_{n-1}\right) +\gamma \beta \rho _{\theta }\left(
u_{m-1},Su_{m-1}\right) \\
&=&\alpha \rho _{\theta }\left( u_{n-1},u_{m-1}\right) +\beta \rho _{\theta
}\left( u_{n-1},u_{n}\right) +\gamma \beta \rho _{\theta }\left(
u_{m-1},u_{m}\right) \\
&\leq &\alpha \rho _{\theta }\left( u_{n-1},u_{m-1}\right) +\left( \beta
\left( \frac{\alpha +\gamma }{1-\beta }\right) ^{n-1}+\gamma \left( \frac{%
\alpha +\gamma }{1-\beta }\right) ^{m-1}\right) \rho _{\theta }\left(
u_{1},u_{0}\right) \text{.}
\end{eqnarray*}%
It follows from Lemma \ref{A2} that $\left\{ u_{n}\right\} $ is a Cauchy
sequence. So, from the completeness of $E$, we obtain that there exists a
point $u\in E$ such that $u_{n}\rightarrow u$. Now, we show that $u$ is a
fixed point of $S$, i.e., $\rho _{\theta }\left( u,Su\right) =0$. Since%
\begin{eqnarray*}
\rho _{\theta }\left( u,Su\right) &\leq &\theta \left( u,Su\right) \left[
\rho _{\theta }\left( u,u_{n+1}\right) +\rho _{\theta }\left(
u_{n+1},u_{n+2}\right) +\cdots +\rho _{\theta }\left(
u_{n+v-1},u_{n+v}\right) +\rho _{\theta }\left( u_{n+v},Su\right) \right] \\
&\leq &\theta \left( u,Su\right) \left[ \rho _{\theta }\left(
u,u_{n+1}\right) +\rho _{\theta }\left( u_{n+1},u_{n+2}\right) +\cdots +\rho
_{\theta }\left( u_{n+v-1},u_{n+v}\right) +\rho _{\theta }\left(
Su_{n+v-1},Su\right) \right] \\
&\leq &\theta \left( u,Su\right) \left[ \rho _{\theta }\left(
u,u_{n+1}\right) +\rho _{\theta }\left( u_{n+1},u_{n+2}\right) +\cdots +\rho
_{\theta }\left( u_{n+v-1},u_{n+v}\right) \right. \\
&&\left. +\alpha \rho _{\theta }\left( u_{n+v-1},u\right) +\beta \rho
_{\theta }\left( u_{n+v-1},u_{n+v}\right) +\gamma \rho _{\theta }\left(
u,Su\right) \right] ,
\end{eqnarray*}%
we have%
\begin{eqnarray*}
\left( 1-\gamma \theta \left( u,Su\right) \right) \rho _{\theta }\left(
u,Su\right) &\leq &\theta \left( u,Su\right) \left[ \rho _{\theta }\left(
u,u_{n+1}\right) +\rho _{\theta }\left( u_{n+1},u_{n+2}\right) +\cdots
\right. \\
&&+\left. \rho _{\theta }\left( u_{n+v-1},u_{n+v}\right) +\alpha \rho
_{\theta }\left( u_{n+v-1},u\right) +\beta \rho _{\theta }\left(
u_{n+v-1},u_{n+v}\right) \right] .
\end{eqnarray*}%
Since $\Gamma _{1}<\frac{1}{\Gamma _{2}}$, we get $\left( 1-\gamma \theta
\left( u,Su\right) \right) \in \left[ 0,1\right) $. So, it follows from (\ref%
{2.5}) and convergence of $\left\{ u_{n}\right\} $\ that $\rho _{\theta
}\left( u,Su\right) =0$. This means that $u$ is a fixed point of $S.$ Now,
we need to show that $u$ is a unique fixed point. Let assume that there
exist another fixed point $v$. Then, we have 
\begin{eqnarray*}
\rho _{\theta }\left( u,v\right) &=&\rho _{\theta }\left( Su,Sv\right) \leq
\alpha \rho _{\theta }\left( u,v\right) +\beta \rho _{\theta }\left(
u,Su\right) +\delta \rho _{\theta }\left( v,Sv\right) \\
&=&\alpha \rho _{\theta }\left( u,v\right) .
\end{eqnarray*}%
Since $\alpha <1$, we obtain that $\rho _{\theta }\left( u,v\right) =0,$
i.e., $u$ is the unique fixed point of $S$.
\end{proof}

\begin{remark}
In Theorem \ref{A4}, if we take $\theta \left( u,w\right) =s$ for all $%
u,w\in E$ where $s\geq 1$, then we derive Theorem 2.4 of \cite{mr}.
\end{remark}

In Reich fixed point theorem, \ if we get $\alpha =0,$ then we obtain
following generalized Kannan fixed point theorem in $b_{v}\left( \theta
\right) $ metric spaces.

\begin{theorem}
\label{A5}Let $E$ be a complete $b_{v}\left( \theta \right) $ metric space
and $S$ a mapping on $E$ satisfying:%
\begin{equation*}
\rho _{\theta }\left( Su,Sw\right) \leq \beta \rho _{\theta }\left(
u,Su\right) +\gamma \rho _{\theta }\left( w,Sw\right)
\end{equation*}%
for all $u,w\in E$ where $\beta $ and $\gamma $ are nonnegative constants
with $\beta +\gamma <1$ and $\Gamma _{1}<\frac{1}{\Gamma _{2}}$ where $%
\Gamma _{1}=\min \left\{ \beta ,\gamma \right\} $ and $\Gamma _{2}=\max
\left\{ \theta \left( u,Su\right) ,\theta \left( Su,u\right) \right\} $.
Then $S$ has a unique fixed point.
\end{theorem}

\begin{remark}
In Theorem \ref{A5},

\begin{enumerate}
\item if $v=1$ and $\theta \left( u,w\right) =1$ for all $u,w\in E$ where $%
s\geq 1$, then we obtain Kannan fixed point theorem \cite{kan} in complete
usual metric spaces.

\item if $v=2$ and $\theta \left( u,w\right) =s$ for all $u,w\in E$ where $%
s\geq 1$, then we derive Theorem 2.4 of \ \cite{geo}.

\item if $v=2$ and $\theta \left( u,w\right) =1$ for all $u,w\in E$ where $%
s\geq 1$, then we obtain main theorem of \cite{das} without the assumption
of orbitally completeness of the space and the main theorem of \cite{ak}.
\end{enumerate}
\end{remark}

\end{document}